\documentclass{article}
\usepackage{amsfonts}
\usepackage{amsmath}
\usepackage{wasysym}
\usepackage{enumitem}

\newtheorem{theorem}{Theorem}[section]

\newtheorem{corollary}[theorem]{Corollary}
\newtheorem{lemma}[theorem]{Lemma}

\def \squareforqed{\hbox{\rlap{$\sqcap$}$\sqcup$}}
\def \qed{\ifmmode\squareforqed\else{\unskip\nobreak\hfil
          \penalty50\hskip1em\null\nobreak\hfil\squareforqed
          \parfillskip=0pt\finalhyphendemerits=0\endgraf}\fi}
\newenvironment{proof}{\noindent{\bf Proof} }{\qed\medskip}

\newenvironment{example}{\begin{quote}{\textbf{Example} }}{\end{quote}}

\def\notrelation#1#2{\ooalign{$\hfil#1\mkern1mu/\hfil$\crcr$#1#2$}}
\def\notdiv{{\mathpalette\notrelation{\,|}}}

\def\ZZ{{\mathbb Z}}
\def\NN{{\mathbb N}}
\def\QQ{{\mathbb Q}}
\def\RR{{\mathbb R}}

\def\Qmax{Q_{\mathrm{max}}}
\def\Mmax{M_{\mathrm{max}}}
\def\Mgood{M_{\mathrm{good}}}
\def\Mbad{M_{\mathrm{bad}}}
\def\Acrit{A_{\mathrm{crit}}}
\def\Amax{A_{\mathrm{max}}}
\def\Anext{A_{\mathrm{next}}}

\newcommand{\num}{\mathop{\mathrm{num}}\nolimits}
\newcommand{\den}{\mathop{\mathrm{den}}\nolimits}
\newcommand{\HRR}{\mathop{\mathrm{HRR}}\nolimits}

\newcommand\eg{\textit{e.g.}}
\newcommand\ie{\textit{i.e.}}

\title{Fault-Tolerant Modular Reconstruction of Rational Numbers}
\author{John Abbott}
\date{1st May 2015}

\begin{document}
\maketitle

\begin{abstract}

  In this paper we present two efficient methods for reconstructing a
  rational number from several residue-modulus pairs, some of which
  may be incorrect.  One method is a natural generalization of that
  presented by Wang, Guy and Davenport in~\cite{WGD1982} (for
  reconstructing a rational number from \textit{correct} modular
  images), and also of an algorithm presented by Abbott
  in~\cite{Abb1991} for reconstructing an \textit{integer} value from
  several residue-modulus pairs, some of which may be incorrect.  We
  compare our heuristic method with that of B\"ohm, Decker, Fieker and
  Pfister~\cite{BDFP2012}.

\end{abstract}

\medskip
\noindent
\textbf{Keywords:} fault-tolerant rational reconstruction, chinese remaindering

\section{Introduction}

The problem of intermediate expression swell is well-known in computer
algebra, but has been greatly mitigated in many cases by the use of modular
methods.  There are two principal techniques: those based on the
\textit{Chinese Remainder Theorem}, and those based on \textit{Hensel's
  Lemma}.  In this paper we consider only the former approach.

Initially modular methods were used in cases where integer values were sought
%cite KNUTH????????
(\eg~for computing GCDs of polynomials with integer coefficients); the answer
was obtained by a direct application of the Chinese Remainder Theorem.  Then
in 1981 Wang presented a method allowing the reconstruction of
\textit{rational numbers}~\cite{Wan1981} from their modular images: the
original context was the computation of partial fraction decompositions.
Wang's idea was justified in a later paper~\cite{WGD1982} which isolated the
rational number reconstruction algorithm from the earlier paper.  More
recently, Collins and Encarnaci\'on~\cite{CoEn1994} corrected a mistake in
Wang's paper, and described how to obtain an especially efficient
implementation.  Wang's method presupposes that all residue-modulus pairs are
correct; consequently, the moduli used must all be coprime to the denominator
of the rational to be reconstructed.

%The Chinese Remainder methods are also well suited to parallel computation:
%essentially the same computation must be performed once for each modulus.
A well-known problem of modular methods is that of \textbf{bad reduction}:
this means that the modular result is not correct for some reason.
Sometimes it will be obvious when the modular result is bad (and these can
be discarded), but other times it can be hard to tell.  The
\textit{Continued Fraction Method} for the fault-tolerant reconstruction of
\textit{integer values} when some of the modular images may be bad was
presented in~\cite{Abb1991}.

In this paper we consider the problem of reconstructing a rational number
from its modular images allowing for some of the modular images to be
erroneous.  We combine the corrected version of Wang's algorithm with the
\textit{Continued Fraction Method}.  Our resulting new \textit{FTRR}
Algorithm (see section~\ref{FTRR}) reconstructs rational numbers from
several modular images allowing some of them to be bad.  The \textit{FTRR}
Algorithm contains both old methods as special cases: when it is known that
all residues are correct we obtain Wang's method (as corrected
in~\cite{CoEn1994}), and if the denominator is restricted to being~$1$ then
we obtain the original \textit{Continued Fraction Method}.
%  Unlike Wang's method our \textit{FTRR}
% algorithm allows the reconstruction of rationals whose denominators have a
% factor in common with some of the moduli, though any such moduli must
% necessarily count as bad ones.
Finally, we note that the correction highlighted in~\cite{CoEn1994} is a
natural and integral part of our method.

Our \textit{FTRR} Algorithm gives a strong guarantee on its result: if a
suitable rational exists then it is unique and the algorithm will find it;
conversely if no valid rational exists then the algorithm says so.  However,
the uniqueness depends on bounds which must be given in input, including an
upper bound for the number of incorrect residues.  Since this information is
often not known in advance, we present also the \textit{HRR} Algorithm (see
section~\ref{HRR})~---~it is a heuristic reconstruction technique based on
the sample principles as \textit{FTRR}.  This heuristic variant is much
simpler to apply since it requires only the residue-modulus pairs as input.
It will find the correct rational provided the correct modular images
sufficiently outnumber the incorrect ones; if this is not the case then
\textit{HRR} will usually return an \textit{indication of failure} but it may
sometimes reconstruct an incorrect rational.

In section~\ref{comparison} we briefly compare our \textit{HRR} algorithm
with the \textit{Error Tolerant Lifting} Algorithm presented
in~\cite{BDFP2012} which is based on lattice reduction, and which serves
much the same purpose as \textit{HRR}.  We mention also some
\textit{combinatorial} reconstruction schemes (presented in~\cite{Abb1991})
which can be readily adapted to perform fault tolerant rational
reconstruction.

\subsection{Envisaged Setting}
\label{envisaged-setting}

We envisage the computation of one or more rational numbers
(\eg~coefficients of a polynomial) by \textit{chinese remainder style}
modular computations where not all cases of bad reduction can be detected.
If we know upper bounds for numerator and denominator, and also for the
number of bad residue-modulus pairs then we can apply the \textit{FTRR}
algorithm of section~\ref{FTRR}.  Otherwise we apply the \textit{HRR}
algorithm of section~\ref{HRR}.  Naturally, in either case we require that
the bad residue-modulus pairs are not too common.
%  The moduli should be
% coprime, though there is no requirement that they actually be prime.

When using \textit{FTRR} we use the sufficient precondition
(inequality~(\ref{FTRR-precondition})) to decide whether more residue-modulus
pairs are needed; when we have enough pairs we simply apply the
reconstruction algorithm to obtain the answer.

When using \textit{HRR}, we envisage that the computation is organized as
follows.  Many modular computations are made iteratively, and every so
often an attempt is made to reconstruct the sought after rational
number(s).  If the attempt fails, further iterations are made.  If the
attempt succeeds then a check is made of the ``convincing correctness'' of
the reconstructed rational (see step~(4) of Algorithm \textit{HRR}); if the rational is not ``convincing'' then again
further iterations are made.

\medskip

The perfect reconstruction algorithm would require only the minimum number
of residue-modulus pairs (thus not wasting ``redundant'' iterations), and
never reconstructs an incorrect rational (thus not wasting time checking
``false positives'').  Our \textit{HRR} algorithm comes close to having
both characteristics.

% \section{IDEAS, CHANGES, etc}

% Present basic algm assuming moduli are prime.
% Observe that powers of primes work just as well.
% Observe that (coprime) composite moduli are equiv to the CRT of
% several prime moduli.

% Maybe observe that non-coprime moduli can be made to work; needs change to
% integer CRT.  What to do if there is an ``incompatibility''? Two possibilities:
% (A) discard the whole GCD between the moduli
% (B) discard only part e.g. (1 mod 6) and (4 mod 6) are incompat, but
% gcd(4-1,6) = 3, so they are both compat with (1 mod 3) -- worth it?

\section{Notation and Assumptions}

We are trying to reconstruct a rational number, $p/q$, from many
residue-modulus pairs: $x_i \bmod m_i$ for $i=1,2,\ldots,s$.  For each
index~$i$ satisfying $q x_i \equiv p \bmod{m_i}$ we say that $x_i$ is a
\textbf{good residue} and $m_i$ is a \textbf{good modulus}; otherwise, if
the equivalence does not hold, we call them a \textbf{bad residue} and a
\textbf{bad modulus}.

For simplicity, we assume that the moduli $m_i$ are pairwise coprime: this
assumption should be valid in almost all applications.  For clarity of
presentation, it will be convenient to suppose that the moduli are labelled
in increasing order so that $m_1 < m_2 < \cdots < m_s$.  For our algorithms
to work well it is best if the moduli are all of roughly similar size;
otherwise, in an extreme situation where there is one modulus which is larger
than the product of all the other moduli, reconstruction cannot succeed if
that one large modulus is bad.

\smallskip
We say that a rational $p/q$ is \textbf{normalized} if $q > 0$ and $\gcd(p,q)=1$.

\subsection{Continued Fractions}

Here we recall a few facts about continued fractions; proofs and further
properties may be found in~\cite{HW1979}, for instance.

Let $x \in \RR$; then~$x$ has a unique representation as a
\textbf{continued fraction}:
\begin{equation}
x = [ a_0, a_1, a_2, \ldots ] = a_0 + \frac{1}{a_1+\frac{1}{a_2+\frac{1}{\ldots}}}
\end{equation}
where all $a_j \in \ZZ$; for $j > 0$
the integers $a_j$ are positive, and are called \textbf{partial quotients}.  If $x
\in \QQ$ then there are only finitely many partial quotients; otherwise
there are infinitely many.

We define the $k$-th \textbf{continued fraction approximant to~$x$} to be the
rational $r_k/s_k$ whose continued fraction is $[a_0, a_1, \ldots, a_k]$.
These \textit{approximants} give ever closer approximations to $x$, that is the
sequence $|x - r_k/s_k|$ is strictly decreasing.  We also have that:
\begin{equation} \label{cfa-num-den-growth}
\begin{array}{l}
\phantom{\sum} %just to increase vertical separation
a_k \, r_{k-1} \quad \le \quad r_k \quad < \quad (a_k+1)\, r_{k-1} \\
\phantom{\sum}
a_k \, s_{k-1} \quad \le \quad s_k \quad < \quad (a_k+1)\, s_{k-1}
\end{array}
\end{equation}

We recall here Theorem~184 from~\cite{HW1979} which will play a crucial role.
\begin{theorem}
Let $x \in \RR$ and $\frac{r}{s} \in \QQ$.  If $|x - \frac{r}{s}| < \frac{1}{2s^2}$ then
$\frac{r}{s}$ appears as a continued fraction approximant to~$x$.
\end{theorem}

\section{Main Proposition}

Our main proposition provides the key to reconstructing a rational from a \textit{single}
residue-modulus pair, $X \bmod M$.

\begin{theorem} \label{main-thm}

  Let $X \bmod M$ be a residue-modulus pair; thus $X,M \in \ZZ$ with $M \ge
  2$.  Let $P, Q \in \NN$ be positive bounds for numerator and
  denominator respectively.  Suppose there exists a factorization $M =
  \Mgood \, \Mbad \in \NN$ such that $2 P Q \Mbad^2 < M$, and suppose also that
  there exists a rational $p/q \in \QQ$ with $|p| \le P$ and $1 \le q \le
  Q$ which satisfies $p \equiv qX \bmod \Mgood$.  Then $\frac{p}{q}$ is
  unique, and is given by
  $$\frac{p}{q} = X - M \cdot \frac{R}{S}$$
  where $\frac{R}{S}$ is the last continued fraction approximant to $\frac{X}{M}$
  with denominator $\le Q \Mbad$; moreover, the next approximant has
  denominator $> \Mgood/2 |p|$.

\end{theorem}

\begin{proof}

By hypothesis we have $p = qX - k \Mgood$ for some $k \in \ZZ$.
Dividing by $qM$ we obtain:
\begin{equation}\label{eq1}
\frac{p}{qM} = \frac{X}{M} - \frac{k}{q \Mbad}
\end{equation}
We shall write $\frac{R}{S}$ for the normalized form of $\frac{k}{q \Mbad}$; thus
$\gcd(R,S)=1$ and $0 < S \le q \Mbad \le Q \Mbad$.

The condition $2 P Q \Mbad^2 < M$ implies that $|p| < \frac{\Mgood}{2 Q \Mbad}$.
We use this to estimate how well $\frac{R}{S}$ approximates $\frac{X}{M}$:

$$ \left| \frac{X}{M} - \frac{R}{S} \right| \quad=\quad \frac{|p|}{qM}
   \quad<\quad  \frac{1}{2 q Q \Mbad^2} \quad\le\quad \frac{1}{2S^2} $$

Applying Theorem~184 from~\cite{HW1979} we see that $\frac{R}{S}$ is indeed
one of the continued fraction approximants for $\frac{X}{M}$.  Next we
show that $\frac{R}{S}$ is the \textit{last} approximant with denominator
$\le Q \Mbad$.

We start by showing that if $\frac{r}{s}$ is
any rational number with $1 \le s \le \frac{\Mgood}{2 |p|}$ and different from $\frac{R}{S}$ then
$\left| \frac{X}{M} - \frac{r}{s} \right| \ge \left| \frac{X}{M} - \frac{R}{S} \right|$.
First note that:
%%% no other rational with denominator less than or
%%%equal to $\Mgood/2 |p|$ is closer to $\frac{X}{M}$.  Let  Then
$$
\left| \frac{r}{s} - \frac{R}{S} \right| \quad\ge\quad
\frac{1}{sS} \quad\ge\quad \frac{2 |p|}{\Mgood} \cdot \frac{1}{q \Mbad} \quad=\quad
\frac{2 |p|}{q M}
$$
Whence $\left| \frac{X}{M} - \frac{r}{s} \right| \ge
\left| \frac{r}{s} - \frac{R}{S} \right| - \left| \frac{R}{S} - \frac{X}{M} \right| \ge
\frac{|p|}{q M} = \left| \frac{X}{M} - \frac{R}{S} \right|$.
Therefore, any approximant coming after $\frac{R}{S}$, and hence closer to $\frac{X}{M}$, must have
denominator $> \Mgood/2 |p|$.

The claim that $\frac{p}{q} = X-M \cdot \frac{R}{S}$ follows immediately from equation~(\ref{eq1}).

%% MUST ALSO PROVE THAT CANNOT GET A DIFFERENT RESULT USING A DIFFERENT FACTORIZATION OF $M$!!!
%% PROOF: since $p/q$ is determined by last cont frac approximant with ``small'' denom, it is
%% unique provided $\Mbad$ is within a reasonable range.
\end{proof}

\begin{corollary} \label{largest-partial-quotient}

  Let $j$ be the index of the approximant $\frac{R}{S}$ in
  Theorem~\ref{main-thm}.  Let $\Mgood^* = \gcd(p - qX, M)$, and
  $\Mbad^* = \frac{M}{\Mgood^*}$.  Let $\Qmax$ be the greatest integer
  strictly less than $\frac{\Mgood^*}{2 |p| \Mbad^*}$.
  Then the $(j+1)$-th partial quotient is at least $\frac{\Qmax}{q}-1$.  If
  $\Mgood^* \ge 2 |p| q \Mbad^* \bigl(\max(2 |p|,q) \Mbad^* + 2 \bigr)$ then the
  $(j+1)$-th partial quotient is the largest of all.

\end{corollary}

\begin{proof}

  Observe that $\Mgood^* \ge \Mgood$ and $\Mbad^* \le \Mbad$
  regardless of the original factorization $M = \Mgood \Mbad$ used in the
  theorem.

  By applying the theorem with $P = |p|$ and $Q = q$, and using the
  factorization $M=\Mgood^* \Mbad^*$ we see that $S \le q \Mbad^*$;
  furthermore the $(j+1)$-th approximant has denominator greater than
  $\Mgood^*/2 |p| > \Qmax \Mbad^*$.
% By applying the proposition once again with the same factorization but with $P
%  = |p|$ and $Q = \Qmax$ we see that the $(j+1)$-th approximant has
%  denominator greater than $\Qmax M^*_{bad}$.
  Thus by the final inequality of formula~(\ref{cfa-num-den-growth})
  the $(j+1)$-th partial quotient must be at least $\frac{\Qmax}{q} -1$.

  Since~$S$, the denominator of the $j$-th approximant, is at most $q \Mbad^*$ no
  partial quotient with index less than or equal to~$j$ can exceed $q \Mbad^*$.
  Also, since the denominator of the $(j+1)$-th approximant is greater than
  $\Mgood^*/2 |p|$ and the denominator of the final approximant is at most~$M$,
  every partial quotient with index greater than $j+1$ is less than $2 |p| \Mbad^*$.

  The hypothesis relating $\Mgood^*$ to $\Mbad^*$ thus guarantees that the $(j+1)$-th
  partial quotient is the largest.
\end{proof}

\begin{example}

  Let $X=7213578109$ and $M=101 \times 103 \times 105 \times 107 \times
  109$.  Let $P=Q=100$.  By magic we know that $\Mbad = 101$, so we seek
  the last approximant to $\frac{X}{M}$ with denominator at most $Q \Mbad =
  10100$.  It is the $10$-th approximant and has value $\frac{R}{S} =
  2116/3737$.  Hence the candidate rational is $\frac{p}{q} = X - M \cdot
  \frac{R}{S} = \frac{13}{37}$ which does indeed satisfy the numerator and
  denominator bounds.  The next approximant has denominator $9701939$, and as
  predicted by the theorem this is greater than $\Mgood/2|p| \approx 4851359$.
  The next partial quotient is $2596 > \frac{\Qmax}{q} -1 \approx 1297$ as predicted
  by the corollary.

\end{example}

\section{The Fault Tolerant Rational Reconstruction Algorithm}
\label{FTRR}

We present our first algorithm for reconstructing rational numbers based on
Theorem~\ref{main-thm}.  The algorithm expects as inputs:
\begin{itemize}
\setlength{\itemsep}{-3pt}
\item  a set of
  residue-modulus pairs $\{ x_i \bmod m_i : i = 1,\ldots,s \}$,

\item  upper bounds
$P$ (for the numerator), and $Q$ (for the denominator) of the rational to be
reconstructed,

\item an upper bound $e$ for the number of bad residue-modulus pairs.
\end{itemize}

We recall that the moduli $m_i$ are coprime, and are in increasing order so that $m_1 <
m_2 < \cdots < m_s$.  We define $\Mmax = m_{s-e+1} m_{s-e+2} \cdots m_s$, the product
of the~$e$ largest moduli;
this implies that $\Mbad \le \Mmax$ and $\Mgood \ge m_1 m_2 \cdots m_s$.
Thus to be able to apply Theorem~\ref{main-thm} we require that
% We shall use the following \textit{measure of maximum information loss}:
% let $\errmax = m_{s-e+1} m_{s-e+2} \cdots m_s$ the product of the $e$
% largest moduli.  We use this measure to obtain a \textbf{sufficient
%   precondition} to permit this reconstruction method to succeed: we require
% that
\begin{equation} \label{FTRR-precondition}
M=m_1 m_2 \cdots m_s > 2 P Q \Mmax^2
\end{equation}
% \smiley {\small \smiley} {\scriptsize \smiley} {\tiny \smiley}
Comparing this with the condition given in~\cite{WGD1982} we see that an extra
factor of $\Mmax^2$ appears: this is to allow for a loss of information
``up to $\Mmax$'', and to allow for an equivalent amount of redundancy
requisite for proper reconstruction.  If the denominator bound $Q = 1$ then
the precondition~(\ref{FTRR-precondition}) simplifies to that for
the \textit{Continued Fraction Method}~\cite{Abb1991}.

\subsection{The FTRR Algorithm}

The main loop of this algorithm is quite similar to that in~\cite{WGD1982}: it just
runs through the continued fraction approximants for $X/M$, and selects the last
one with ``small denominator''; there is a simple final computation to produce the answer.

\begin{description}
\setlength{\itemsep}{0pt}

\item [1] Input $e$, $P$, $Q$, and $\{ x_i \bmod m_i : i=1,\ldots,s \}$
\item [2] If $x_i \equiv 0 \bmod m_i$ for at least $s-e$ indices~$i$ then return 0.
\item [3] Set $M = \prod_i m_i$.  Compute integer $X$ satisfying $X \equiv x_i \bmod m_i$
for each~$i$ (via Chinese remaindering).
\item [4] Compute $\Mmax = m_{s-e+1} m_{s-e+2} \cdots m_s$.
\item [5] If $\gcd(X,M) > P \Mmax$ then return \textsc{failure}.
\item [6] Put $u = (1, 0, M) \in \ZZ^3$ and $v = (0, 1, X) \in \ZZ^3$.
\item [7] While $|v_2| \le Q \Mmax$ do
\item [7.1] $q = \lfloor u_3/v_3 \rfloor$
%%\item [7.2] $w = u-qv$; \quad $u = v$; \quad $v = w$
\item [7.2] $u = u-qv$; swap $u \longleftrightarrow v$
\item [8] Set $r = X+M \cdot \frac{u_1}{u_2}$ as a normalized rational.
\item [9] Check whether $r$ is a valid answer:\\
\ie~$|\num(r)| \le P$ and $\den(r) \le Q$ and at most~$e$ bad moduli.
\item [10] If $r$ is valid, return $r$; otherwise return \textsc{failure}.
\end{description}

\noindent
\textbf{Note} that in the algorithm the successive values of $-\frac{u_1}{u_2}$ at
the end of each iteration around the main loop are just the continued
fraction approximants to $X/M$.

\begin{example}
For some inputs to algorithm \textit{FTRR} there is no valid answer.
If the input parameters are $e=0$, $P=Q=1$ and $x_1=2$ with
modulus $m_1 = 5$ then with the given bounds the only possible valid
answers are $\{-1,0,1\}$ but $2 \bmod 5$ does not correspond to any of
these~---~the check in step~(9) detects this.
\end{example}

%-------------------------------------------------------
\subsection{Correctness of FTRR Algorithm}

We show that \textit{FTRR} finds the right answer if it exists,
and otherwise it produces \textsc{failure}.

\medskip

We first observe that if the correct result is~$0$ and at most~$e$
residue-modulus pairs are faulty then step~(2) detects this, and rightly
returns~$0$.  We may henceforth assume that the correct answer, if it exists, is a
non-zero rational $\frac{p}{q} \in \QQ$ with $|p| \le P$ and $1 \le q \le Q$.

\begin{lemma}
If there is a valid non-zero solution $p/q$ then $\gcd(X,M) \le P \Mmax$.
\end{lemma}
\begin{proof}
As the $m_i$ are coprime $\gcd(X,M) = \prod_i \gcd(x_i,m_i)$.  If the modulus
$m_i$ is good then $\gcd(x_i,m_i) \,|\, p$; conversely, if $\gcd(x_i,m_i)
\,\notdiv\, p$ then $m_i$ is a bad modulus.  Hence $\prod_{m_i
  \hbox{\scriptsize good}} \gcd(x_i,m_i) \,|\, p$; while $\prod_{m_i
  \hbox{\scriptsize bad}} \gcd(x_i,m_i) \le \prod_{m_i \hbox{\scriptsize
    bad}} m_i \le \Mmax$.  It is now immediate that $\gcd(X,M) \le P
\Mmax$.
\end{proof}

From the lemma we deduce that the check in step~(5) eliminates only $(X,M)$
pairs which do not correspond to a valid answer.  We also observe that for
all $(X,M)$ pairs which pass the check in step~(5) the denominator of the
normalized form of $X/M$ is at least $2 Q \Mmax$, so the loop exit
condition in step~(7) will eventually trigger.

\medskip

The values~$X$ and~$M$ computed in step~(3) are precisely the
corresponding values in the statement of Theorem~\ref{main-thm}.
However, we do not know the correct factorization $M = \Mgood \Mbad$;
but since there are at most~$e$ bad residue-modulus pairs we do know that
$\Mbad \le \Mmax$, and this inequality combined with the
requirement~(\ref{FTRR-precondition}) together imply that $2 P Q \Mbad^2 <
M$ so we may apply the proposition.  
%Clearly $\Mgood \ge M/\Mmax$;
%hence $\Mgood/2p \ge \Mgood/2P \ge M/(2 P \Mmax) > Q \Mmax$.
Thus the algorithm simply has to find the last continued fraction
approximant $\frac{R}{S}$ with denominator not exceeding $Q \Mmax$,
which is precisely what the main loop does: at the end of each iteration
$-\frac{u_1}{u_2}$ and $-\frac{v_1}{v_2}$ are successive approximants
to $\frac{X}{M}$, and the loop exits when $|v_2| > Q \Mmax$.

So when execution reaches step~(8), the fraction $-\frac{u_1}{u_2}$ is
precisely the approximant $\frac{R}{S}$ of the proposition.  Thus
step~(8) computes the candidate answer in $r$, and step~(9) checks
that the numerator and denominator lie below the bounds~$P$ and~$Q$,
and that there are no more than~$e$ bad moduli.  If the checks pass,
the result is valid and is returned; otherwise the algorithm reports
\textsc{failure}.

\subsection{Which Residues were Faulty?}
\label{identify-bad-moduli}

Assume the algorithm produced a normalized rational $p/q$, and we want to
determine which moduli (if any) were faulty.  We could simply check which
images of $p/q$ modulo each $m_i$ are correct.  However, there is another,
more direct way of identifying the bad moduli: we show that the bad $m_i$
are exactly those which have a common factor with~$S$, that is the final
value of $u_2$.

If $m_i$ is a good modulus then we have $\gcd(m_i, q)=1$ because otherwise
if the gcd were greater than~$1$ then $p \equiv q x_i \bmod m_i$
implies that the gcd divides~$p$ too, contradicting the assumption
that~$p$ and~$q$ are coprime.

Multiplying equation~(\ref{eq1}) from the proof of Theorem~\ref{main-thm}
by $q M$ we obtain $p = q X - M \cdot \frac{R}{S}$ whence $M \cdot \frac{R}{S}$ is an integer.
By definition of a bad modulus $m_i$ we must have $M\cdot \frac{R}{S} \not\equiv 0 \bmod m_i$.
Since $m_i \,|\, M$, we must have $\gcd(m_i, S) > 1$.

\section{The Heuristic Algorithm}
\label{HRR}

The main problem with the \textit{FTRR} Algorithm is that we do not generally
know good values for the input bounds $P,Q$ and~$e$.  In this
\textit{heuristic} variant the only inputs are the residue-modulus pairs; the
result is either a rational number or an indication of \textsc{failure}.  The
algorithm is heuristic in that it may (rarely) produce an incorrect result,
though if sufficiently many residue-pairs are input (with fewer than $\frac{1}{3}$ of
them being bad) then the result will be correct.

\subsection{Algorithm HRR: Heuristic Rational Reconstruction}

\begin{description}
\item [1] Input $x_i \bmod m_i$ for $i=1,\ldots,s$.  Set $\Acrit = 10^{6}$ (see note below).
\item [2] Put $M = \prod_i m_i$.  Compute $X \in \ZZ$ such that $|X| < M$ and $X \equiv x_i \bmod m_i$ via Chinese remaindering.
\item [3] If $\gcd(X,M)^2 > \Acrit M$ then return $0$.
\item [4] Let $\Amax$ be the largest partial quotient in the continued fraction of $X/M$.\\
          If $\Amax < \Acrit$ then return \textsc{failure}.
\item [5] Put $u = (1, 0, M) \in \ZZ^3$ and $v = (0, 1, X) \in \ZZ^3$, and set $q=0$.
\item [6] While $q \not= \Amax$ do
\item [6.1] $q = \lfloor u_3/v_3 \rfloor$
\item [6.2] $u = u-qv$; swap $u \longleftrightarrow v$
%\item [7] If $\left|\frac{X}{M} + \frac{u_1}{u_2} \right| \ge \frac{1}{2 u_2^2}$ then continue to the next iteration.
%\item [7] Set $\Mbad = gcd(M,u_2)$. % if $2 N D F_{bad}^2 < \varepsilon M$ then return $N/D$.
\item [7] Return $N/D$ the normalized form of $X + M u_1/u_2$; we could also return $\Mbad = \gcd(M,u_2)$.
\end{description}

The idea behind the algorithm is simply to exploit
Corollary~\ref{largest-partial-quotient} algorithmically.  This corollary
tells us that, provided $\Mgood$ is large enough relative to $\Mbad$, we
can reconstruct the correct rational from the last approximant before the
largest partial quotient.  Moreover, if the proportion of residue-modulus
pairs which are bad is less than $\frac{1}{2}$ then $\Mgood$ will eventually become
large enough (provided the moduli are all roughly the same size).
% Corollary~\ref{size-of-largest-partial-quotient} tells us that the largest
% partial quotient will grow unboundedly as we accumulate more
% residue-modulus pairs (and provided that there are more good pairs than
% bad).

Since zero requires special handling, there is a special check in step~(3) for this
case.  The heuristic will produce zero if ``significantly more than half of the
residues'' are zero~---~strictly this is true only if all the moduli are prime and of
about the same magnitude.

\subsubsection*{The role of $\Acrit$}

To avoid producing too many \textit{false positives} we demand that the largest
partial quotient be greater than a certain threshold, namely $\Acrit$.
The greater the threshold, the less likely we will get a false positive;
but too great a value will delay the final recognition of the correct value.
The suggested value $\Acrit = 10^6$ worked well in our trials.

\subsubsection*{Alternative criterion for avoiding false positives}

Our implementation in CoCoALib~\cite{CoCoALib} actually uses a slightly different
``convincingness criterion'' in step~(4).  Let $\Amax$ be the largest partial
quotient, and $\Anext$ the second largest.  Our alternative criterion is
to report \textsc{failure} if $\Amax/\Anext$ is smaller than a given
threshold~---~in our trials a threshold value of $4096$ worked well, but
our implementation also lets the user specify a different threshold.
%In practice, the two criteria appear to be similarly effective.

\subsubsection{Complexity of HRR}

Under the natural assumption that each residue satisfies
$| x_i | \le m_i$, we see that the overall complexity of algorithm
\textit{HRR} is $O \bigl( (\log M)^2 \bigr)$, the same as for Euclid's algorithm.
Indeed the chinese remaindering in step~(2) can be done with a modular
inversion (via Euclid's algorithm) and two products for each modulus.  The computation of
the partial quotients in step~(4) is Euclid's algorithm once again.  And
the main loop in step~(6) is just the extended Euclidean algorithm.

We note that the overall computational cost depends on how often
\textit{HRR} is called in the envisaged lifting loop (see
subsection~\ref{envisaged-setting}).  Assuming that the moduli chosen are
all about the same size, a reasonable compromise approach is a
``geometrical strategy'' where \textit{HRR} is called whenever the number
of main iterations reaches the next value in a geometrical progression.
This compromise avoids excessive overlifting and also avoids calling
\textit{HRR} prematurely too often.  The overall cost of \textit{HRR} with
such a strategy remains $O \bigl( (\log M)^2 \bigr)$ where~$M$ here denotes the combined
modulus in the final, successful call to \textit{HRR}.

In practice, if the cost of calling \textit{HRR} is low compared to the
cost of one modular computation in the main loop then it makes sense to
call \textit{HRR} more frequently.  The geometrical strategy should begin only
when (if ever) the cost of a call to \textit{HRR} is no longer relatively
insignifiant.

\subsection{Simultaneous Rational Reconstruction}

In~\cite{BS2011} the authors presented an interesting algorithm for
the simultaneous reconstruction of several rational numbers having a
``small common denominator''; moreover, in certain cases the algorithm
would require a remarkably low combined modulus~---~smaller than the
product of numerator and denominator of some of the reconstructed
rationals.  Nevertheless, it is not obvious how that algorithm could
be modified to handle bad moduli.

The case of simultaneous reconstruction arises, for instance, when the
final result is a vector or polynomial.  While each component of the vector or
each coefficient of the polynomial could be reconstructed separately,
we can do slightly better: for example, we normally expect to find the
same bad moduli for each component/coefficient, and it often happens that
there is a ``small common denominator''.

We outline how \textit{HRR} can be used to reconstruct several
rationals simultaneously; of course, \textit{HRR} can be replaced by
another ``single rational'' reconstruction algorithm.  For simplicity
we assume that the Chinese remaindering has already been done, so we
have a single common modulus~$M$ and several residues $X_1, \ldots,
X_k$ each one corresponding to a rational number to be reconstructed.

\begin{description}
\setlength{\itemsep}{0pt}

\item [1] Input $X_1, \ldots, X_k \in \ZZ$ and the common modulus $M \in \ZZ$ with $M>2$.
\item [2] Set $D=1$; this will be our common denominator.
\item [3] For $i=1,\ldots,k$ do
\item [3.1] Apply \textit{HRR} to $D\,X_i$ and~$M$; if this fails, return \textsc{failure}.
\item [3.2] Let $\Mbad$ be the bad modulus factor; replace $M = M/\Mbad$.
\item [3.3] Let $R/S$ be the reconstructed rational, and set $d = \gcd(R,D)$.
\item [3.4] Set $q_i = \frac{R/d}{SD/d}$.
\item [3.5] Set $D = SD$, the new common denominator.
\item [4] Return $q_1,\ldots,q_k$.
\end{description}

\noindent
Notes:
\begin{itemize}
\setlength{\itemsep}{0pt}
\item Step~(3.2) is useful only if the bad moduli for each coefficient are
      essentially the same; if this is not the case, it can be skipped.
\item The order of the $X_i$ is potentially important; if some of the coefficients
      are expected to be simpler than others then the simpler ones should appear
      at the start~---~\eg~it often happens tha the coefficients of the highest and
      lowest degree terms in a polynomial are simpler than the ``central'' coefficients.
\item Though not strictly necessary, it is probably worth reducing the $X_i$ modulo
      the updated value of $M$ in step~(3.2).
\item In practice it may be useful to return the bad modulus factors found (even in
      the case of \textsc{failure}).
\end{itemize}

\begin{example}
Let $M=12739669845=101 \times 103 \times 105 \times 107 \times 109$, and let $X_1 = -5790759020$, $X_2 = -2410207808$ and $X_3 = -9484324233$.

We start with the common denominator $D=1$.
On iteration $i=1$,
we compute $\HRR(X_1,M) = 5/11$ with no bad moduli.  So we set $q_1=5/11$ and update $D = 11$.

On iteration $i=2$, we compute $\HRR(D\, X_2, M) = 209/37$ with no bad moduli.  So we set $q_2 = \frac{209/37}{D} = 19/37$, and update $D = 407$.

On iteration $i=3$, we compute $\HRR(D \, X_3, M) = 204$ with no bad moduli.  So we set $q_3 = \frac{204}{D} = 204/407$; there is no need to update $D$ since \textit{HRR} produced an integer.

The final answer is $(q_1, q_2, q_3) = (5/11, 19/37, 204/407)$.  Note
that attempting to compute $q_3$ directly by calling $\HRR(X_3,M)$
fails; indeed, multiplying by the common denominator when we computed
$\HRR(D \, X_3, M)$ has let us reconstruct a more complex rational
than we could obtain by direct reconstruction.  This also highlights
the fact that the success of the reconstruction can depend on the order
of the residues $X_i$; had $X_3$ been the first residue the algorithm
would have failed (because the modulus $M$ is ``too small'').
\end{example}

\section{Comparison with Other Methods}
\label{comparison}

\subsection{Reconstruction via Lattice Reduction}

A reconstruction technique based on 2-dimensional lattice reduction is
presented as Algorithm~6 \textit{Error Tolerant Lifting}
(abbr.~\textit{ETL}) in~\cite{BDFP2012}.  This algorithm is similar in
scope to our \textit{HRR}, and not really comparable to our \textit{FTRR}
algorithm (which needs extra inputs from the user).

In practice there are two evident differences between \textit{ETL} and our
\textit{HRR}.  The first is that \textit{ETL} produces many more \textit{false
positives} than \textit{HRR}; our refinement (\textbf{B}) below proposes a way to
rectify this.  The second is that \textit{ETL} finds balanced rationals
more easily than unbalanced ones, \ie~it works best if the numerator and
denominator contain roughly the same number of digits.  For balanced
rationals, \textit{ETL} and \textit{HRR} need about the same number of
residue-modulus pairs; for unbalanced rationals \textit{ETL} usually
needs noticeably more residue-modulus pairs than \textit{HRR}.

% algorithm, their approach does not require as input data the bounds for the
% numerator and denominator nor for the number of faulty residues (\ie~our
% parameters $P,Q$ and $e$); instead their method imposes an intrinsic bound
% on the numerator and denominator (see step~(7) of their algorithm), and it
% is unclear whether the number of bad moduli is limited.

\subsubsection{Practical Refinements to \textit{ETL}}

We propose two useful refinements to \textit{ETL} as it is described
in~\cite{BDFP2012}.

\begin{description}
\item[A]  We believe that a final checking step should be added
to the \textit{ETL} algorithm so that it rejects results where half or more
of the moduli are bad.  Consider the following example: the moduli are
$11,13,15,17,19$ and the corresponding residues are $-4,-4,-4,1,1$.  The
rather surprising result produced by \textit{ETL} is~$1$; it seems
difficult to justify this result as being correct.  Here we see explicitly
the innate tendancy of \textit{ETL} to favour ``trusting'' larger moduli
over smaller ones.

\item [B] The aim of this second refinement is to reduce the number of false
positives which \textit{ETL} produces.  We suggest replacing their
acceptance criterion $a_{i+1}^2 + b_{i+2}^2 < N$ by a stricter condition
such as $a_{i+1}^2 + b_{i+2}^2 < N/100$.  This change may require one or
two more ``redundant'' residue-modulus pairs before \textit{ETL} finds the
correct answer, but it does indeed eliminate most of the false positives.
\end{description}

% Despite appearances the \textit{ETL} algorithm does not minimise the sum of
% the squares of numerator and denominator for a given number of bad moduli.
% For instance with the moduli $11,13,15,17,19$ and the corresponding
% residues $7,1,0,14,12$ our \textit{FTRR} algorithm (with $P=47$ and $Q=20$)
% reconstructs the rational $45/19$ having bad modulus~$19$ whereas
% \textit{ETL} reconstructs the significantly more complex value $-185/192$
% having bad modulus~$15$.

% If we have some information about the relative sizes of numerator and
% denominator of the number to be reconstructed then \textit{FTRR} can
% use this information to obtain a successful reconstruction using fewer
% moduli than \textit{ETL} needs.  For instance, with moduli $11,13,15,17,19$
% and corresponding residues $7,8,8,4,2$ the \textit{ETL} algorithm
% produces the rational $13/16$ for which both~$13$ and~$15$ are bad moduli.
% In contrast, applying \textit{FTRR} with $P=47$ and $Q=20$ we obtain the
% rational $-43/19$ with only one bad modulus, namely~$19$.

\subsubsection{Comparison of Efficiency}

We define the \textbf{efficiency} of a reconstruction method to be the
logarithm of the combined modulus when reconstruction first succeeds.
Our trials involved reconstructing rationals from a succession of
residue-modulus pairs where the moduli were all about the same size;
so the efficiency is essentially proportional to the number of pairs
required for reconstruction to succeed.  For simplicity, we shall use
the number of pairs as our measure here.

We claim that the \textit{efficiency} is the most appropriate measure
of how well the algorithm performs because the computational cost of
obtaining a new residue-modulus pair (potentially the result of a
lengthy computation such as a Gr\"obner basis) generally far exceeds
the cost of attempting reconstruction, so counting the number of pairs
needed gives a good estimate of actual total computational cost.  This
point of view is valid provided the rational to be reconstructed does
not have especially large numerator or denominator.

\medskip
We have implemented \textit{HRR} and \textit{ETL} in CoCoALib~\cite{CoCoALib} and CoCoA-5~\cite{CoCoA}.
Using these implementations we compared the efficiency of
\textit{HRR} and \textit{ETL} by generating a random rational $N/D$ (with a
specified number of bits each for the numerator and denominator), and then
generating the modular images $x_i \bmod m_i$ where the $m_i$ run through
successive primes starting from~$1013$.  Note that in this first trial there are
no bad residue-modulus pairs.  We then counted how many residue-modulus
pairs were needed by the algorithms before they were able to reconstruct
the original rational.

We then repeated the experiment but this time, with probability 10\%, each
residue was replaced by a random value to simulate the presence of bad
residues.  As expected, the number of residue-modulus pairs needed for
successful reconstruction increased by about 25\%.

In each case the successful reconstruction took less than $0.1$ seconds.

\begin{center}
\begin{tabular}{|l|c|c|c|c|}
\hline
%              & $2000/0$ bits & $1600/400$ bits & $1200/800$ bits & $1000/1000$ bits \\
              & $\phantom{\int_q^b}\frac{2000}{0}$ bits & $\frac{1600}{400}$ bits & $\frac{1200}{800}$ bits & $\frac{1000}{1000}$ bits \\
\hline
\textit{HRR} 0\% bad   &   190  &  191  &  190 &  190 \\
\textit{ETL} 0\% bad   &   361  &  293  &  224 &  189 \\
\hline
\textit{HRR} 10\% bad  &   244  &  236  &  246 &  244 \\
\textit{ETL} 10\% bad  &   457  &  375  &  283 &  242 \\
% \hline
% \textit{HRR} 20\% bad  &   305  &  292  &  284 &  300 \\
% \textit{ETL} 20\% bad  &   545  &  453  &  338 &  299 \\
\hline
\end{tabular}
\end{center}

Observe that the number of pairs needed by \textit{HRR} is essentially
constant, while \textit{ETL} matches the efficiency of \textit{HRR} only
for perfectly balanced rationals; as soon as there is any disparity between
the sizes of numerator and denominator \textit{HRR} becomes significantly
more efficient.

\subsection{Combinatorial Methods}
\label{combinatorial}

It is shown in~\cite{Sto1963} that reconstruction of integers by Chinese
Remaindering is possible provided no more than half of the
\textit{redundant residues} are faulty.  The correct value is identified
using a \textit{voting system} (see~\cite{Sto1963} for details).  We can
extend the idea of a voting system to allow it to perform fault tolerant
rational reconstruction: the only difference is that for each subset of
residue-modulus pairs we effect an exact rational reconstruction (rather
than an exact integer reconstruction).  However the problem of poor
computational efficiency remains.
%... if at least half the subsets produce
%the same rational then we return that value, otherwise we return an
%indication of failure.  

An elegant and efficient scheme for fault-tolerant chinese remaindering for
integers was given in~\cite{Ram1983}; however the method is valid only for at
most one bad modulus.  Several generalizations of Ramachandran's
scheme were given in~\cite{Abb1991}; however, these are practical really
only for at most~$2$ bad moduli.  Like the voting system, these schemes
could be easily adapted to perform fault-tolerant rational reconstruction,
but in the end the \textit{Continued Fraction Method} (upon which \textit{FTRR} is
based) is more flexible and more efficient.

\section{Conclusion}

We have presented two new algorithms for solving the problem of fault
tolerant rational reconstruction, \textit{FTRR} and \textit{HRR}.  The
former is a natural generalization both of the original rational
reconstruction algorithm~\cite{WGD1982} and of the fault tolerant integer
reconstruction algorithm~\cite{Abb1991}.  The latter is a heuristic variant
which is easier to use in practice since it does not require certain bounds
as input.

Our \textit{HRR} algorithm and the \textit{ETL} algorithm
from~\cite{BDFP2012} offer two quite distinct (yet simple) approaches to the
same problem.  They have comparable practical efficiency when
reconstructing balanced rationals, whereas \textit{HRR} is usefully more
efficient when reconstructing unbalanced rationals.

\end{document}